\documentclass[a4paoer, 12pt]{amsart}
\usepackage{amsmath,amscd, mathtools}
\usepackage{amssymb}
\usepackage{amsxtra}
\usepackage{enumerate}
\usepackage[mathscr]{eucal}
\usepackage[margin=2.1cm]{geometry}

\usepackage{hyperref}
\usepackage{amsmath}
\usepackage{amsthm}
\usepackage{amssymb}
\usepackage{tikz}
\usepackage{float}
\usepackage{mathrsfs}
\usepackage{xcolor}

\usepackage[hypcap=false]{caption}
\usepackage{caption}
\usepackage{pst-node}
\usepackage{tikz-cd}

\newtheorem{theorem}{Theorem}[section]
\newtheorem{lemma}[theorem]{Lemma}
\newtheorem{proposition}[theorem]{Proposition}
\newtheorem{corollary}[theorem]{Corollary}

\theoremstyle{definition}
\newtheorem{example}[theorem]{Example}
\newtheorem{remark}[theorem]{Remark}

\newcommand{\thmref}[1]{Theorem~\ref{#1}}
\newcommand{\lemref}[1]{Lemma~\ref{#1}}

\newcommand{\eqnref}[1]{~{\textrm(\ref{#1})}}

\usepackage{hyperref}

\usepackage{graphicx}
\usepackage{subfig}
\graphicspath{ {./images/} }

\numberwithin{equation}{section}

\usetikzlibrary{decorations.pathreplacing,decorations.markings}
\newcounter{flipflop}
\tikzset{
on each straight segment/.style={
    decorate,
    decoration={
        show path construction,
        moveto code={},
        lineto code={
            \path [#1]
            (\tikzinputsegmentfirst) -- (\tikzinputsegmentlast);
        },
        curveto code={
            \path  (\tikzinputsegmentfirst)
            .. controls
            (\tikzinputsegmentsupporta) and (\tikzinputsegmentsupportb)
            ..
            (\tikzinputsegmentlast);
        },
        closepath code={
            \path 
            (\tikzinputsegmentfirst) -- (\tikzinputsegmentlast);
        },
    },
},
set flipflop/.code=\setcounter{flipflop}{#1},
on each other straight segment/.style={
    decorate,
    decoration={
        show path construction,
        moveto code={},
        lineto code={\stepcounter{flipflop}
            \path \ifodd\value{flipflop} [#1]\fi
            (\tikzinputsegmentfirst) -- (\tikzinputsegmentlast);
        },
        curveto code={
            \path  (\tikzinputsegmentfirst)
            .. controls
            (\tikzinputsegmentsupporta) and (\tikzinputsegmentsupportb)
            ..
            (\tikzinputsegmentlast);
        },
        closepath code={
            \path 
            (\tikzinputsegmentfirst) -- (\tikzinputsegmentlast);
        },
    },
},
mid arrow/.style={postaction={decorate,decoration={
            markings,
            mark=at position .5 with {\arrow[#1]{stealth}}
}}},
}

\begin{document}

\title[$z$-classes of Palindromic Automorphisms]{On the $z$-classes of Palindromic automorphisms of Free Groups}

\author[K. Gongopadhyay]{Krishnendu Gongopadhyay}

\author[L. Kundu]{LOKENATH KUNDU}

\author[S. V. Singh]{Shashank Vikram Singh}

\address{
Indian Institute of Science Education and Research (IISER) Mohali,
		Knowledge City, Sector 81, SAS Nagar, Punjab 140306, India}
	\email{krishnendu@iisermohali.ac.in}

\address{Indian Institute of Science Education and Research (IISER) Mohali,
		Knowledge City, Sector 81, SAS Nagar, Punjab 140306, India.}
	\email{lokenath@iisermohali.ac.in}

 \address{
Indian Institute of Science Education and Research (IISER) Mohali,
		Knowledge City, Sector 81, SAS Nagar, Punjab 140306, India}
	\email{shashank@iisermohali.ac.in}

\makeatletter
\@namedef{subjclassname@2020}{\textup{2020} Mathematics Subject Classification}
\makeatother



\begin{abstract}
  The palindromic automorphism group is a subgroup of the automorphism group $Aut(F_n).$ We establish a necessary and sufficient condition for a matrix in $GL_n(\mathbb{Z})$ representing a palindromic automorphism of $F_n.$ We prove that the number of the $z$-classes in $\Pi A(F_n)$ is infinite. We further classify the conjugacy classes of the reducible palindromic automorphisms.
\end{abstract}

\maketitle           



\section{Introduction}
Given any group $G$, two elements $x$ and $y$ in $G$ are called $z$-equivalent or belong to the same \emph{$z$-classes}, denoted $x \sim_z y$, if their centralizers $Z_G(x)$ and $Z_G(y)$ are conjugate in $G$. The $z$-classes give a disjoint partition of the group $G$. The conjugacy classes of centralizers appeared at many places in the literature, e.g., classical work of Borel-de Siebenthal \cite{Bo} and Steinberg \cite{St}. In general, a group may be infinite and it may have infinitely many conjugacy classes, but the number of $z$-classes is often finite, see \cite{GS},\cite{St}. It motivates Kulkarni to view the $z$-classes as an algebraic ingredient to classify `dynamical types' in a geometry. In \cite{Kulkarni1}, Kulkarni has explained the motivation in detail and has proposed a systematic study of the conjugacy classes of the centralizers in a group. Kulkarni has also initiated the terminology `$z$-class' in \cite{Kulkarni1}. For subsequent work to understand the $z$-classes in groups, see \cite{Bh, Go, Kulkarni, Ku3}. For further information on $z$-classes, we refer to the survey \cite{Bh1}. 

Recently, Parsad \cite{parsad} has proved that if two elements in the mapping class group of an orientable genus $g$ surface are 
$z$-equivalent, then their Nielsen-Thurston type is the same. Due to the analogy between several aspects of the mapping class groups and the free group (outer) automorphisms, cf. \cite{bv}, one can ask for the $z$-classes for $Aut(F_n)$. In this paper, we ask this question for a special subgroup of $Aut(F_n)$, viz., the subgroup of all palindromic automorphisms, which we recall below. 

 \medskip Let $F_n$ be a free group with a generating set $S:=\lbrace a_1,a_2,\ldots,a_n \rbrace$. Let $w=a_{i1}a_{i2}\cdots a_{ik}\cdots a_{in}$ be a reduced word in $F_n$ where $a_{ik}$ is the letter at $k$-th position of $w$, and $a_{i1},~a_{i2},~\cdots ,a_{ik}~\in ~S$. Define $\bar{w}=a_{in}\cdots a_{ik}\cdots a_{i2}a_{i1}$ to denote the word obtained by reversing the order of the letters in $w$. A word $w$ is said to be a palindromic word if $w=\bar{w}$. An automorphism that maps each generator of $F_n$ to a palindromic word of $F_n$ is called a palindromic automorphism. The collection of palindromic automorphisms forms an infinite index subgroup of $Aut(F_n)$, denoted by $\Pi A(F_n)$. Collins introduced the concept of palindromic automorphisms in \cite{Collins}. He has also proved a finite presentation of the palindromic automorphism group. Geometric and algebraic properties of the palindromic automorphism group were further investigated in \cite{KG, Glover, center, PR}. We recall from \cite{Glover} that the palindromic automorphism group $\Pi A(F_n)$ is generated by the following automorphisms:
\begin{itemize}
    \item Automorphisms $A_{ij}=(a_i||a_j),~i\neq j,$ which maps $a_i\mapsto a_ja_ia_j$ and fix all other generators $a_k,~k\neq i$.
\item Automorphisms $\sigma_{a_i}$ which maps $a_i\mapsto a_i^{-1}$ and fix all other generators $a_k,~k\neq i$.
\item Automorphisms $\tau_{\rho}$ corresponding to elements of the symmetric group $S_n$ which permute the 
    $a_1,a_2,\dots,a_n$
 among themselves. That is, $\tau_{\rho}(a_i)=a_{\rho(i)}$ for $\rho \in S_n$.
\end{itemize}

\medskip 
The center of the palindromic automorphism group $\Pi A(F_n) ~\text{for }n\geq 3$ has been computed in \cite{KG}, \cite{center}. Following these works, it is a natural question to ask for the classification of the $z$-classes in $\Pi A(F_n)$ for $n\geq 3$.  The classification of the $z$-classes would provide a partition of $\Pi A(F_n)$, and that may provide internal ingredients towards understanding the dynamical types of the palindromic automorphisms.

First note that there is a well-known correspondence $\psi$ between $Aut(F_n)$ and the group of integer matrices $GL_n(\mathbb Z)$, cf. \cite[Chapter 3]{Bogopolski}. The map $\psi: Aut(F_n)\rightarrow  GL_n(\mathbb{Z})$ is the \emph{abelianization} map defined by the following rule: 

\medskip
For any $\alpha\in Aut(F_n)$ the $(i,j)$-th entry of the matrix $\psi(\alpha)$ is equal to the sum of exponents of the letter $a_j$ in $\alpha(a_i)$. Then the map $\psi$ is an epimorphism. The matrix $\psi(\alpha)$ is called the abelianization matrix corresponding to $\alpha$ and denoted by $\widehat{\alpha}$.  

\medskip 
Let $\widehat{ GL_n}(\mathbb Z)$ denote the subgroup of $GL_n(\mathbb Z)$ which consists of all matrices $\widehat A$ such that each row and each column will have exactly one odd entry. It is worth noting that $\widehat{GL}_n(\mathbb{Z})$ is the preimage of the symmetric group $S_n$, consisting of all permutation matrices, under the modulo $2$  reduction map $\pi: GL_n(\mathbb{Z}) \to GL_n(\mathbb{Z}/\langle 2 \rangle)$. We prove the following. 
\begin{theorem}\label{ncase}
The image of $\Pi A(F_n)$ under $\psi$ is $\widehat{ GL_n(\mathbb Z)}$. 
\end{theorem}

 For \( n = 3 \), we employ the correspondence between the groups \( \Pi A(F_3) \) and \( \widehat{GL_3(\mathbb{Z})} \) to identify the images of the generators of \( \Pi A(F_3) \) within \( \widehat{GL_3(\mathbb{Z})} \). Subsequently, we determine the centralizers of these images and thereby deduce their \( z \)-classes. By constructing explicit examples, we establish that the number of \( z \)-classes in \( \widehat{GL_3(\mathbb{Z})} \) is infinite, and consequently, the number of \( z \)-classes in \( \Pi A(F_3) \) is also infinite (see Theorem~\ref{P3}). The analysis for \( n = 3 \) yields the following theorem.

\begin{theorem}\label{zclassn}
There are infinitely many $z$-classes in the group $\Pi A(F_n)$ for $n\geq 3$. 
\end{theorem}

It is worth noting here that the general structure of the conjugacy classes in $\Pi A(F_n)$ seems not so well-understood, and we are not aware of any results in this direction. The classification of conjugacy classes is often a first step to the classification of the $z$-classes. Accordingly, we ask this problem for $\Pi A(F_3)$. We use the same idea to actually ask for the conjugacy classes in $\widehat{GL_3(\mathbb Z)}$. We have not been able to provide a full solution for this problem. However, for certain subclasses of reducible palindromic automorphisms, that is, those whose $\psi$-images are reducible, we have determined their conjugacy classes; see Theorem \Ref{sim1}. It would be interesting to obtain a full understanding of the conjugacy classes and the $z$-classes of palindromic automorphisms and their $\psi$-images. 

\medskip 
The structure of the paper is as follows. In Section \ref{proof of theorem 1.1}, we prove the correspondence between $\Pi A(F_n)$ and $\widehat{GL_n(\mathbb Z)}$. Section~\ref{proof of theorem 1.2} is primarily devoted to the proof of Theorem~\ref{zclassn}, whereas Section~\ref{reducible palindromic} addresses the conjugacy classes of a certain family of reducible palindromic automorphisms.
 \subsection{Acknowledgment}
 Gongopadhyay acknowledges the research grants CRG/2022/003680 and DST/INT/JSPS/P-323/2020. Kundu acknowledges support from the grant CRG/2022/003680. Singh acknowledges support from the NBHM PDF number 0204/27/(29)/2023/R$\&$D-II/11930 and from the IISER Mohali appointment number IISER/23/Dean Faculty/2976.

We would like to thank the anonymous referee for pointing out mistakes in the proof of Theorem \ref{ncase}. We acknowledge the valuable comments of Mr. Rahul Mondal and express our gratitude to Dr. Sarbendu Rakshit and Dr. Pabitra Barman for their assistance during the computational experiments.

\section{Proof of the Theorem \ref{ncase}} \label{proof of theorem 1.1}

\medskip Let $F_n=\langle a_1,a_2,a_3,\dots,a_n\rangle$ and $\widehat{A} \in \widehat{GL_n(\mathbb{Z})}$.
We use the following short exact sequence from \cite{njf}:\begin{equation}\label{seqn}
    1\rightarrow \mathcal{PI}_n\rightarrow P \Pi A(F_n) \xrightarrow {\psi} \Gamma_n[2]\rightarrow 1,
\end{equation} where ${\mathcal{PI}}_n$ is the palindromic Torelli group, $P \Pi A(F_n)$ is the pure palindromic automorphism group generated by palindromic automorphisms $A_{ij}$ and $\sigma_{a_i}$, and  $$\Gamma_n[2]:=\{ A \in GL_n(\mathbb{Z})|~A \equiv I_n \pmod 2\},$$ 
where $I_n$ denotes $n \times n$ identity matrix, that is,  $\Gamma_n[2]$ is the principle congruence subgroup of level $2$, and $\psi$ is the restriction of the abelianization map described earlier. 

First, we claim that $\widehat{GL_n(\mathbb{Z})} \subseteq \psi(\Pi A(F_n))$. Suppose $\widehat{A} \in \widehat{GL_n(\mathbb{Z})}$. The proof of the claim is divided into the two following cases.

\begin{itemize}
    \item [Case 1.]\label{cs1} 
    If the diagonal entries of $\widehat{A}$ are odd and the non-diagonal entries are all even integers, then $\widehat{A}\in \Gamma_n[2]$. So via the above short exact sequence (see Equation (\ref{seqn})) there exists a palindromic automorphism $\alpha \in P\Pi A(F_n)$ such that $\psi(\alpha)=\widehat{A}$.
    
    \item[Case 2.] If $\widehat A$ does not satisfy the conditions of Case~1, 
then $\widehat A \bmod 2$ is a permutation matrix $P \in GL_n(\mathbb Z_2)$. 
Choose $\rho \in S_n$ whose permutation matrix equals $P^{-1}$. 
Then $\psi(\tau_\rho)\widehat A \equiv I_n \pmod 2$, so 
$\psi(\tau_\rho)\widehat A \in \Gamma_n[2]$. 
   
   Now, using the surjectivity of Equation (\ref{seqn}), we have a pure palindromic automorphism, say $\alpha'$, such that $\psi(\alpha')=\psi(\tau_{\rho})\widehat A$. We now proceed to compute $$\psi(\tau_{\rho^{-1}})\psi(\alpha')=\psi(\tau_{\rho^{-1}}) \psi(\tau_{\rho})\ \widehat A.$$
    Since $\psi$ is a homomorphism, we have $\psi(\tau_{\rho^{-1}} \circ \alpha')=\psi(\tau_{\rho^{-1}} \circ \tau_{\rho})\widehat A$, where $\circ$ denotes the composition of two palindromic automorphisms. Hence, we have a palindromic automorphism $\tau_{\rho^{-1}} \circ \alpha'$, say $\alpha$, satisfying $\psi(\alpha)=\widehat{A}.$
\end{itemize}
This proves our claim.

For the converse, let $\alpha \in \Pi A(F_n)$. Observe that the palindromic word $\alpha(a_i)$ cannot be of the form $w\Bar{w}$ for any $a_i$. Indeed, if $\alpha(a_i)=w\bar w$, then $\psi(\alpha) \notin GL_n(\mathbb{Z})$ since its determinant would not be $\pm 1$. Hence, for each $a_i$, the word $\alpha(a_i)$ must be a palindromic word of odd length, and therefore has the form $wa\bar{w}$, where $a \in \{a_1^{\pm 1}, a_2^{\pm 1},\dots,a_n^{\pm 1}\}$. Consequently, the exponent of $a$  is odd while the exponents of the remaining basis elements in $\alpha(a_i)$ are even. Since $\psi(\alpha)\in  GL_n(\mathbb{Z})$, this argument implies that $\psi(\alpha)\in \widehat{GL_n(\mathbb{Z})}$.\qed

\section{Proof of Theorem \thmref{zclassn}} \label{proof of theorem 1.2}

We begin this section with several elementary yet significant observations concerning the group $\Pi A(F_3)$. These observations enable the identification of infinitely many $z$-classes in the groups $\Pi A(F_n)$. Here, we begin by identifying the $z$-classes of the generators of $\Pi A(F_3)$. 

\begin{lemma}\label{case1}
For $1 \le i \neq j \le 3$, the abelianization matrices $\psi(A_{ij})$ are conjugate to one another in $\widehat{GL_3(\mathbb{Z})}$. Consequently, their centralizers in $\widehat{GL_3(\mathbb{Z})}$ are conjugate.
\end{lemma}

\begin{proof}
Recall that for $i \neq j$, the matrix $\psi(A_{ij})$ has the form
\[
\psi(A_{ij}) = I_3 + 2E_{ij},
\]
where $E_{ij}$ denotes the $3 \times 3$ matrix whose $(i,j)$-entry is $1$ and all other entries are $0$.  

Let $\sigma \in S_3$ be a permutation and $P_{\sigma}$ the corresponding permutation matrix, so that $P_{\sigma} e_k = e_{\sigma(k)}$ for the standard basis vectors $\{e_1,e_2,e_3\}$. Each permutation matrix lies in $\widehat{GL_3(\mathbb{Z})}$, hence conjugation by $P_{\sigma}$ is an inner automorphism of $\widehat{GL}_3(\mathbb{Z})$.

A straightforward computation gives
\[
P_{\sigma} E_{ij} P_{\sigma}^{-1} = E_{\sigma(i)\sigma(j)}.
\]
Therefore,
\[
P_{\sigma}\, \psi(A_{ij})\, P_{\sigma}^{-1}
  = P_{\sigma}(I_3 + 2E_{ij}) P_{\sigma}^{-1}
  = I_3 + 2 P_{\sigma}E_{ij}P_{\sigma}^{-1}
  = I_3 + 2E_{\sigma(i)\sigma(j)}
  = \psi(A_{\sigma(i)\sigma(j)}).
\]
Hence, for any $\sigma \in S_3$, $\psi(A_{\sigma(i)\sigma(j)})$ is conjugate to $\psi(A_{ij})$.  
In particular, all $\psi(A_{ij})$ (for $i \neq j$) lie in a single conjugacy class under permutation matrices.

Since conjugation preserves centralizers, we have
\[
Z(\psi(A_{\sigma(i)\sigma(j)})) = P_{\sigma} Z(\psi(A_{ij})) P_{\sigma}^{-1}.
\]
Thus, the centralizers of the matrices $\psi(A_{ij})$ are conjugate to one another in $\widehat{GL_3(\mathbb{Z})}$.
\end{proof}

\begin{theorem}{\label{z-classes}}
The $z$-classes of images of the generators of $\Pi A(F_3)$ under the map $\psi$ are as follows:
    \begin{enumerate}
    \item Centralizers of the elements $\psi(A_{ij})$ for $1\leq i\neq j \leq 3$ are conjugate to each other by $\psi (\tau_{\rho})$  in $\widehat{GL_3(\mathbb{Z})}$ for some $\rho \in S_3$.
    \item Centralizers of the elements $\psi(\sigma_{a_i})$  are also conjugate to each other by $\psi(\tau_{\rho})$  in $\widehat{GL_3(\mathbb{Z})}$ for some $\rho \in S_3$.
    \item Let $S_3=\langle(12),(132)\rangle$ be generated by a 2-cycle and a $3$-cycle. Centralizers of the images of the automorphisms corresponding to the generators of $S_3$ under $\psi$ are not conjugate in $\widehat{GL_3(\mathbb{Z})}$. 
\end{enumerate}
\end{theorem}
 
\begin{proof} We look for the images of the generators of $\Pi A(F_3)$ in $\widehat{GL_3(\mathbb{Z})}$ and then search for matrices that conjugate the images. These matrices would also conjugate the centralizers of the images. We see that the matrices that are conjugating the images of the generators are the images of the palindromic automorphism corresponding to $\tau_{\rho}$ for some $\rho \in S_3$. 
    \begin{itemize}
    \item [Case 1:] This follows from \lemref{case1}.

   \item[Case 2:] The image of automorphism $\sigma_{a_i}$ for $1 \leq i \leq 3$  is $3\times3$ diagonal matrice in $\widehat{GL_3(\mathbb{Z})}$ whose $ii^{th}$ entry is $-1$ and $1$ at $jj^{th}$ places for $i\neq j$ and $1\leq i,j \leq 3$.
   Here, it is easy to see that  $\psi(\sigma_{a_i})$ and $\psi(\sigma_{a_j})$ for $1 \leq ~i\neq j ~\leq 3$ are conjugate to each other by $\psi(\tau_{(ij)})$  where $(ij) \in S_3$. Therefore, the generators $\sigma_{a_i}$ for $1 \leq i \leq 3$ are conjugate to each other and their centralizers too.

   \item[Case 3:] In this case, we study the $z$-classes of the automorphisms $\tau_{\rho}$ corresponding to the generators $\rho=$(12) and $\rho=$(132) of $S_3$. Let $\widehat{D}$ and $\widehat{E}$ be the matrices corresponding to $\tau_{(12)}$, $\tau_{(132)}$ in $\widehat{GL_3(\mathbb{Z})}$, respectively.  By definition of $\psi$, we have
   $\widehat{D}=\begin{bsmallmatrix}   
        0 & 1 & 0\\      
         1 & 0 & 0\\      
         0 & 0 & 1        
    \end{bsmallmatrix}, 
    \widehat{E}= \begin{bsmallmatrix}
    0 & 0 & 1\\
    1 & 0 & 0\\
    0 & 1 & 0
    \end{bsmallmatrix}.$

   Note that the matrices $\widehat{D}$ and $\widehat{E}$ are not conjugate to each other as their orders are different. Now, we show that the centralizer of the matrix $\widehat{D}$, $Z(\widehat{D})$, is not conjugate to the centralizer of the matrix $E$, $Z(\widehat{E})$. For this purpose, we first compute their centralizers. By computation, we get that $Z(\widehat{D}) \subseteq \widehat{GL_3(\mathbb{Z})}$ contains matrices of the form $\begin{bsmallmatrix}
       a & b & c\\
       b & a & c\\
       u & u & w
   \end{bsmallmatrix}$, and $Z(\widehat{E})\subseteq \widehat{GL_3(\mathbb Z)}$ contains matrices of the form $\begin{bsmallmatrix}
       a & b & c\\
       c & a & b\\
       b & c & a
   \end{bsmallmatrix}$. Now we try to figure out elements of order two in $Z(\widehat{E})$. Suppose
   \begin{center}
      $\begin{bsmallmatrix}
       a & b & c\\
       c & a & b\\
       b & c & a
   \end{bsmallmatrix}
   \begin{bsmallmatrix}
       a & b & c\\
       c & a & b\\
       b & c & a
   \end{bsmallmatrix}=
   \begin{bsmallmatrix}
       a^2+2bc & c^2+2ab & b^2+2ac\\
       b^2+2ac & a^2+2bc & c^2+2ab\\
       c^2+2ab & b^2+2ac & a^2+2bc
   \end{bsmallmatrix}=
   \begin{bsmallmatrix}
       1 & 0 & 0\\
       0 & 1 & 0\\
       0 & 0 & 1
   \end{bsmallmatrix}$
   \end{center}
   By solving the equations $a^2+2bc=1, c^2+2ab=0, b^2+2ac=0$, we get that  $a=\pm1, b=0, c=0$ are the only solutions in $\mathbb{Z}$. Therefore, $\begin{bsmallmatrix}
       -1 & 0 & 0\\
       0 & -1 & 0\\
       0 & 0 & -1
   \end{bsmallmatrix}$ 
   is the only element of order two in $Z(\widehat{E})$. The centralizer $Z(\widehat{D})$ has more than one element of order two. For example
      $\begin{bsmallmatrix}
       -1 & 0 & 0\\
       0 & -1 & 0\\
       0 & 0 & 1
   \end{bsmallmatrix} \text{ and }
   \begin{bsmallmatrix}
       1 & 0 & 0\\
       0 & 1 & 0\\
       0 & 0 & -1
   \end{bsmallmatrix}$. Therefore, the centralizers $Z(\widehat{E})$ and $Z(\widehat{D})$ can not be conjugated to each other.
\end{itemize}
This completes the proof. 
\end{proof}
Further, we observe in Corollary \ref{notsame} that the $z$-classes of $A_{ij}$ and $\sigma_{a_i}$, $1\leq i,j \leq 3$, are not same. We need the following Lemma \ref{infOr2} to prove it.  
\begin{lemma}\label{infOr2}
    The elements of $Z(\psi(A_{ij}))$, $1\leq i,j \leq 3$, are of order 2 or infinity in $\widehat{GL_3(\mathbb Z)}$.
\end{lemma}
\begin{proof}
Theorem \ref{z-classes} guarantees that it is enough to prove it for $A_{12}$. We know that \[ \psi(A_{12}) =\begin{bsmallmatrix}
    1&2&0\\0&1&0\\0&0&1
\end{bsmallmatrix}=\widehat{A}_{12} \text{ (say)}. \] Let \[\widehat{B}=\begin{bsmallmatrix}
    a&b&c\\d&e&f\\g&h&j
\end{bsmallmatrix} \in Z(\psi(A_{12})) \subseteq \widehat{GL_3(\mathbb{Z})} . \] Then  \[ \widehat{A}_{12} \widehat{B} = \widehat{B} \widehat{A}_{12} \implies \begin{bsmallmatrix}
    1&2&0\\0&1&0\\0&0&1
\end{bsmallmatrix} \begin{bsmallmatrix}
    a&b&c\\d&e&f\\g&h&j
\end{bsmallmatrix} = \begin{bsmallmatrix}
    a&b&c\\d&e&f\\g&h&j
\end{bsmallmatrix}  \begin{bsmallmatrix}
    1&2&0\\0&1&0\\0&0&1
\end{bsmallmatrix} \] \[ \implies \begin{bsmallmatrix}
    a + 2 d & b + 2 e & c + 2 f\\d&e&f\\g&h&j
\end{bsmallmatrix}= \begin{bsmallmatrix}
    a & 2 a + b & c \\ d & 2 d + e & f\\ g & 2 g + h & j
\end{bsmallmatrix}. \]   
Comparing the elements, one can easily see that $d=f=g=0,~ e=a$. Hence, $\widehat{B}$ is of the form \[ \widehat{B}=\begin{bsmallmatrix}
    a&b&c\\0&a&0\\0&h&j
\end{bsmallmatrix}. \] Clearly, $\det \widehat{B}=a^2j=\pm 1 \implies (a,j)\in \{(1,1),(1,-1),(-1,1),(-1,-1)\}$.
\begin{itemize}
    \item[Case 1:] When $(a,j)=(1,1)$,\[ \widehat{B}^n=\begin{bsmallmatrix}
    1& \frac{n(2b+ch(n-1))}{2} &nc\\0&1&0\\0&nh&1
\end{bsmallmatrix} \text{ for all } n\ge 1. \] So, $\widehat{B}^n=I \implies b=c=h=0 \implies \widehat{B}=I$. Hence, $\widehat{B}$ is of order infinity.
    \item[Case 2:]  When $(a,j)=(1,-1)$, \[ \widehat{B}^{2n}=\begin{bsmallmatrix}
    1&n(2b+ch)&0\\0&1&0\\0&0&1
\end{bsmallmatrix} \text{ and } \widehat{B}^{2n+1}= \begin{bsmallmatrix}
    1&(2n+1)b+nch&c\\0&1&0\\0&h&-1
\end{bsmallmatrix} \] for all $n\ge 1$. Clearly $\widehat{B}^{2n+1} \neq I$ for all $n\ge 1$. Also, if $\widehat{B}^{2n}=I$ for some $n>1$, then $2b+ch=0 \implies \widehat{B}^2=I$. Then the order of $\widehat{B}$ is 2.
\end{itemize}
 
 The cases $(a,j)=(-1,1)$ and $(a,j)=(-1,-1)$ reduce to Case 2 and Case 1, respectively. This concludes the proof.
\end{proof}
\begin{corollary}\label{notsame}
    The centralizers of $\psi(A_{ij})$ and $\psi(\sigma_{a_i})$, $1\leq i,j \leq 3$, are not conjugate to each other in $\widehat{GL_3(\mathbb{Z})}$.
\end{corollary}
\begin{proof}
    Theorem \ref{z-classes} guarantees that it is enough to prove it for $\psi(A_{12})$ and $\psi(\sigma_{a_1})$ only. We have, \[ \psi(\sigma_{a_1})=\begin{bsmallmatrix}
        -1&0&0\\0&1&0\\0&0&1
    \end{bsmallmatrix}= \widehat{A}_1 ~\text{ (say)}. \] Notice that,  \[ \widehat{B}= \begin{bsmallmatrix}
        1&0&0\\0&0&-1\\0&1&0
    \end{bsmallmatrix} \in \widehat{GL_3(\mathbb Z)}\] and $\widehat{A}_1\cdot \widehat{B} = \widehat{B}\cdot \widehat{A}_1$. So, $\widehat{B}\in Z(\psi(\sigma_{a_1}))$. It is easy to check that the order of $\widehat{B}$ is 4. But Lemma \ref{infOr2} 
 says that $Z(\psi(A_{12}))$ does not contain any element of order 4. So, $Z(\psi(A_{12})) \not\sim Z(\psi(\sigma_{a_1}))$. 
\end{proof}
   \begin{remark}
       In Theorem \ref{z-classes}, we consider the generators of the $\widehat{GL_3(\mathbb{Z})}$ to make computations easier. Observe that one can characterize the $z$-classes of the generators of $\widehat{GL_n(\mathbb{Z})}$ for all $n> 3$. From Theorem \ref{z-classes} and Corollary \ref{notsame}, it follows that the generators $A_{ij}$ and $\sigma_{a_i}$ belong to distinct $z$-classes.
   \end{remark}

   Utilizing the aforementioned observations, we first establish the claim of Theorem \ref{zclassn} in the specific instance where $n=3$. 
   
\begin{theorem}\label{P3}
    The number of $z$-classes of the palindromic automorphism group $\Pi A(F_3)$ is infinite.
\end{theorem}
\begin{proof}

Consider the abelianization matrix 
\[
\widehat{A}_{n,l} =
\begin{bmatrix}
1 & 2n & 0\\
0 & 1 & 2l\\
0 & 0 & 1
\end{bmatrix}
\in \widehat{GL_3(\mathbb{Z})}.
\]
Suppose $\gcd(n,l)=1$. A simple computation shows that
\begin{equation}
Z(\widehat{A}_{n,l})=
\Bigg\{
\begin{bmatrix}
w & rn & c\\
0 & w & rl\\
0 & 0 & w
\end{bmatrix}
\in \psi(\Pi A(F_3))
\ \Big|\ 
r \in \mathbb{Z},\ 
w = \pm 1
\Bigg\}.
\end{equation}

Consider two pairs of nonzero positive integers $(n_1, l_1)$ and $(n_2, l_2)$ satisfying $\gcd(n_i, l_i)=1$ for $i=1,2$. 
Suppose the centralizers $Z(\widehat{A}_{n_1,l_1})$ and $Z(\widehat{A}_{n_2,l_2})$ belong to the same conjugacy class; that is, there exists
$P \in \widehat{GL_3(\mathbb{Z})}$ such that
\[
P\, Z(\widehat{A}_{n_1,l_1})\, P^{-1} = Z(\widehat{A}_{n_2,l_2}).
\]
If 
\[
X =
\begin{bmatrix}
w & r n_1 & c\\
0 & w & r l_1\\
0 & 0 & w
\end{bmatrix}
\in Z(\widehat{A}_{n_1,l_1}),
\]
then for some integers $q, w_1, d$, we have
\[
Y := PXP^{-1} =
\begin{bmatrix}
w_1 & q n_2 & d\\
0 & w_1 & q l_2\\
0 & 0 & w_1
\end{bmatrix}
\in Z(\widehat{A}_{n_2,l_2}).
\]

 Note that $w_1=w$ as conjugate matrices have the same eigenvalues. Let $P = [p_{ij}]_{1\leq i,j\leq 3}$. Computing $PX = YP$ shows that  $P$ is an upper triangular matrix with $p_{ii} = \pm 1$ for all $i$, since $P \in \widehat{GL_3(\mathbb{Z})}$. This computation also yields the following relations:
\begin{align}
p_{11} r n_1 &= p_{22} q n_2, \\
p_{22} r l_1 &= p_{33} q l_2, \\
p_{11} c + p_{12} r l_1 &= p_{23} q n_2 + p_{33} d.
\end{align}

Because $p_{ii} = \pm 1$ for all $i$, we obtain (from the above equations)
\[
q = \frac{\pm r n_1}{n_2} = \frac{\pm r l_1}{l_2}.
\]
If $(n_1, l_1)$ and $(n_2, l_2)$ are pairs of positive integers with $\gcd(n_i, l_i) = 1$, and for some $r \in \mathbb{Z}$ and $X \in Z(\widehat{A}_{n_1,l_1})$ we have
\[
\frac{r n_1}{n_2} \notin \mathbb{Z}
\quad \text{or} \quad
\frac{r l_1}{l_2} \notin \mathbb{Z},
\]
then $Z(\widehat{A}_{n_1,l_1})$ and $Z(\widehat{A}_{n_2,l_2})$ are not conjugate.

Moreover, there exist infinitely many such pairs. In particular, for $j \in \mathbb{N}$, define
\begin{equation}
    \label{condn}
    l_j = 1, \quad r\geq 2, \quad n_{j+1} = rn_j + 1, \quad n_1 = 2 \quad \forall~ j \in \mathbb{N}.
\end{equation}
Then the corresponding centralizers $Z(\widehat{A}_{n_j,l_j})$ are pairwise non-conjugate. This proves the theorem.
\end{proof}

The result obtained for the case $n=3$ extends naturally to the general case. Building on the preceding theorem, we now turn to the proof of Theorem~\ref{zclassn}.
\subsection*{3.1. Proof of Theorem 1.2}

Following the notions of the previous theorem, consider two block diagonal matrices 
$\widehat{A}_k$ and $\widehat{B}_j \in \widehat{GL_n(\mathbb{Z})}$ of the following forms:
\[
\widehat{A}_k =
\begin{bmatrix}
\widehat{A}_{n_k,l_k} & 0 \\
0 & -I_{n-3}
\end{bmatrix},
\qquad
\widehat{B}_j =
\begin{bmatrix}
\widehat{A}_{n_j,l_j} & 0 \\
0 & -I_{n-3}
\end{bmatrix},
\]
where $I_{n-3}$ denotes the $(n-3) \times (n-3)$ identity matrix. 

The eigen values of the matrices $\widehat{A}_{n_k,l_k}$ and $-I_{n-3}$ are $1$ and $-1$ (with multiplicity), respectively. Since the eigenvalues of block diagonal matrices are distinct, the centralizer of $\widehat{A}_k$ has the following form:
\begin{equation*}
\label{eq:centralizers}
\begin{aligned}
Z(\widehat{A}_k) &= 
\left\{
\begin{bmatrix}
\widehat{X} & 0 \\
0 & \widehat{Y}
\end{bmatrix}
\;\middle|\;
\widehat{X} \in Z(\widehat{A}_{n_k,l_k}), \widehat{Y} \in \widehat{GL_{n-3}(\mathbb{Z})}
\right\}. \\[6pt]
\end{aligned}
\end{equation*}
Similarly, we have
\begin{equation*}
\begin{aligned}
Z(\widehat{B}_j) &= 
\left\{
\begin{bmatrix}
\widehat{U} & 0 \\
0 & \widehat{V}
\end{bmatrix}
\;\middle|\;
\widehat{U} \in Z(\widehat{A}_{n_j,l_j}), \widehat{V} \in \widehat{GL_{n-3}(\mathbb{Z})}
\right\}.
\end{aligned}
\end{equation*}

An element  
\[
\begin{bmatrix}
\widehat{X} & 0 \\
0 & \widehat{Y}
\end{bmatrix}
\in Z(\widehat{A}_k)
\quad \text{being conjugate to an element} \quad
\begin{bmatrix}
\widehat{U} & 0 \\
0 & \widehat{V}
\end{bmatrix}
\in Z(\widehat{B}_j)
\]
necessarily implies that $\widehat{X}$ is conjugate to $\widehat{U}$. 
From the preceding theorem, it follows that there exist infinitely many pairs $(n_k, l_k)$ and $(n_j, l_j)$ (in particular, satisfying the condition Equation \eqnref{condn}) for which  
\[
Z(\widehat{A}_{n_k, l_k}) \text{ is not conjugate to } Z(\widehat{A}_{n_j, l_j})
\quad \text{in} \quad \widehat{GL_3(\mathbb{Z})}.
\]
Consequently, there are infinitely many indices $k$ and $j$ such that the corresponding centralizers 
$Z(\widehat{A}_k)$ and $Z(\widehat{B}_j)$ are not conjugate in $\widehat{GL_n(\mathbb{Z})}$. 
This completes the proof.

\section{Conjugacy of  $\psi$-images of reducible palindromic automorphisms of  $F_3$}\label{reducible palindromic}
A matrix \(A \in GL_3(\mathbb{Z})\) is termed reducible, cf. \cite{Bogopolski}, if there exists a permutation matrix \(P \in GL_3(\mathbb{Z})\) such that \(P^{-1}AP\) is a block diagonal matrix of the form
\begin{equation}\label{reduceform}
       P^{-1}AP=\begin{bsmallmatrix}
           X & Y\\ 0 & Z
       \end{bsmallmatrix}.
   \end{equation} 
   A palindromic automorphism is called reducible if its $\psi$-image is reducible. 
\begin{lemma}
\label{rp1}
    A palindromic automorphism matrix 
    $\widehat{A}=\begin{bsmallmatrix}
        a & b & c\\ d & e & f\\ g & h & l
            \end{bsmallmatrix}\in \widehat{GL_3(\mathbb{Z})} $ is reducible if and only if any of the following is true
            \begin{itemize}
                \item $b~=~c~=~0$.
                \item $d~=~f~=~0$.
                \item $g~=~h~=~0$.
                \item $d~=~g~=~0$.
                \item $b~=~h~=~0$.
                \item $f~=~c~=~0$.
            \end{itemize}
    
\end{lemma}
\begin{proof}
   
   Consider a permutation matrix $P_{12}=\begin{bsmallmatrix}
        0 & 1 & 0\\ 1 & 0 & 0\\ 0 & 0 & 1
            \end{bsmallmatrix}$. Then, $P_{12}^{-1}\widehat{A}P_{12}$ is a block diagonal matrix of the above form if and only if $g=h=0$ (in the case when $X$ is a $2\times 2$ matrix as in Equation \eqnref{reduceform}) or $b=h=0$ (in the case when $X$ is a $1\times 1$ matrix). Similarly, other conditions follow from choosing other permutation matrices in $GL_3(\mathbb{Z})$.   
\end{proof}

The following result is well-known. We shall give a proof for completeness. 
    \begin{proposition}
    \label{eg}
        Let $A \in GL_3(\mathbb{Z})$. Then $ +1$ or $-1$ is always an eigenvalue of $A$.
    \end{proposition}

\begin{proof}
    Let $\chi_A(x)$ and $m_A(x)$ be the characteristic and minimal polynomial of $A$, respectively. The degree of the polynomial $\chi_A(x)$ is three (an odd integer). As irrational numbers and complex numbers occur in pairs as roots of $\chi_A(x)$. So, the polynomial $\chi_A(x)$ always has a rational root.  It follows from \textbf{rational root theorem} that the rational root is either $+1$ or $-1$ because the characteristic polynomial $\chi_A(x)$ of $A$ is monic and constant terms are equal to $ +1$ or $-1$.  
\end{proof}

The following remark is an easy consequence of Lemma \ref{rp1} and Proposition \ref{eg}. We use this information to study reducible palindromic automorphisms.

\begin{remark} A reducible palindromic automorphism matrix $\widehat{A}~ \in ~ \widehat{GL_3(\mathbb{Z})}$ is, up to conjugation by a permutation matrix, of one of the following forms: $\begin{bsmallmatrix}
        \pm 1 & 2r & 2s\\ 0 & a & b\\ 0 & c & d
    \end{bsmallmatrix} \text{ or }
    \begin{bsmallmatrix}
      a & b & 2r \\  c & d & 2s \\0 & 0 & \pm 1
    \end{bsmallmatrix}, $ where $ad-bc=\pm 1$.
\end{remark}

 \medskip   \begin{theorem}
   \label{sim1}
       Let $\widehat{A},~\widehat{B}~\in ~\widehat{GL_3(\mathbb{Z})}$ correspond to two reducible palindromic automorphism $A,~B$ respectively. Suppose $\widehat A$ and $\widehat B$ are of the forms  $\begin{bsmallmatrix}
        e & 2r & 2s\\ 0 & a & b\\ 0 & c & d
    \end{bsmallmatrix}$ and $\begin{bsmallmatrix}
        e & 0 & 0\\ 0 & a & b\\ 0 & c & d
    \end{bsmallmatrix}$  respectively, where  $e~=~\pm 1$. Let $g(t)~=~t^2-\tau t+ \delta$ be the characteristic polynomial of $A_2=\begin{bsmallmatrix}
         a & b\\  c & d
    \end{bsmallmatrix}$, and $\delta ~=~\pm 1$. Let $m~=~ e\tau-1-\delta$ and $A_0=A_2-(\tau - e)I_2.$ Now $\widehat{A},~\widehat{B}$ are conjugate in $\widehat{GL_3(\mathbb{Z})}$ if and only if 
   $\begin{bsmallmatrix}
        r & s
    \end{bsmallmatrix} A_0 ~\equiv~ 0~(mod~ m)$ and $|\tau| ~ \neq~2$.
    \end{theorem}

\begin{proof}
If $\widehat{A}$ and $\widehat{B}$ are conjugate in $GL_3(\mathbb{Z})$ then the conjugating matrix is of the form $\widehat{R}~=~\begin{bsmallmatrix}
        u & 2p & 2q \\ 0 & x & y \\ 0 & z & w
    \end{bsmallmatrix}$ $(see ~pp. 166$ in \cite{Onishi}). Since we are working within $\widehat{GL_3(\mathbb{Z})}$, so we assume that $\widehat{R} \in \widehat{GL_3(\mathbb{Z})}$ and that \begin{equation}
        \label{rab}\widehat{R}\widehat{A}=\widehat{B}\widehat{R}.
    \end{equation}
    
    Let $X_2~=~\begin{bsmallmatrix}
        x & y \\ z & w
    \end{bsmallmatrix}$. Then, from Equation \eqnref{rab}, we obtain 
    \begin{equation}
    \label{eq:X2relation}
      \begin{split}
          u \begin{bsmallmatrix}
            2r & 2s 
        \end{bsmallmatrix}+ \begin{bsmallmatrix}
            2p & 2q
        \end{bsmallmatrix}A_2= e \begin{bsmallmatrix}
            2p & 2q
        \end{bsmallmatrix} \text{ and },  A_2X_2=X_2A_2.
      \end{split}  
    \end{equation} From the first relation in Equation \eqref{eq:X2relation}, we have \begin{equation}
      \begin{split}
          u \begin{bsmallmatrix}
            2r & 2s 
        \end{bsmallmatrix}+ \begin{bsmallmatrix}
            2p & 2q
        \end{bsmallmatrix}A_2= e \begin{bsmallmatrix}
            2p & 2q
        \end{bsmallmatrix} \\ \Rightarrow \begin{bsmallmatrix}
            2p & 2q
        \end{bsmallmatrix}(A_2-e.I_2)= - u \begin{bsmallmatrix}
            2r & 2s 
        \end{bsmallmatrix}.
      \end{split}  
    \end{equation}
   Thus      
    \begin{equation}
      \begin{split}
          -u \begin{bsmallmatrix}
            2r & 2s 
        \end{bsmallmatrix} A_0~=~ \begin{bsmallmatrix}
            2p & 2q
        \end{bsmallmatrix}(A_2-e.I_2) A_0.
      \end{split}  
    \end{equation}
      Since $(A_2-e~I_2)A_0~=~m~I_2$, we will have \begin{equation}
          \begin{split}
               -u \begin{bsmallmatrix}
            2r & 2s 
        \end{bsmallmatrix} A_0~=~m  \begin{bsmallmatrix}
            2p & 2q
        \end{bsmallmatrix} I_2\\ \Rightarrow -2u \begin{bsmallmatrix}
            r & s 
        \end{bsmallmatrix} A_0~=~ 2m \begin{bsmallmatrix}
            p & q
        \end{bsmallmatrix} I_2\\ \Rightarrow u \begin{bsmallmatrix}
            r & s 
        \end{bsmallmatrix} A_0 ~\equiv ~0~(mod~m).
          \end{split}
      \end{equation}   
    As $u=\pm 1$, so $\begin{bsmallmatrix}
            r & s 
        \end{bsmallmatrix} A_0 ~\equiv ~0~(mod~m).$ \\
         The converse statement is established by retracing the preceding argument in reverse order, proceeding from the final step back to the initial one.\\
        Assume that $\begin{bsmallmatrix}
        r & s
    \end{bsmallmatrix}A_0~\equiv ~0~(mod~m)$. This implies that $\exists~p,~q~\in ~\mathbb{Z}$ such that $-u \begin{bsmallmatrix}
            2r & 2s 
        \end{bsmallmatrix} A_0~=~m  \begin{bsmallmatrix}
            2p & 2q
        \end{bsmallmatrix} I_2$. Define $\widehat{R}~=~ \begin{bsmallmatrix}
            u & 2p & 2q \\ 0 & a & b \\ 0 & c & d        \end{bsmallmatrix}.$ To prove that $\widehat{R}\widehat{A}~=~\widehat{B}\widehat{R}$ we need to show that \begin{equation}
            \label{pf}
                uR+PA_2~=~eP \text{ where } R~=~\begin{bsmallmatrix}
                2r & 2s
            \end{bsmallmatrix}, \text{ and }  P~=~\begin{bsmallmatrix}
                2p & 2q
            \end{bsmallmatrix}.
            \end{equation}
   Now we have $\begin{bsmallmatrix}
       r & s
   \end{bsmallmatrix} A_0~\equiv~ 0~(mod~m).$ Hence \begin{equation*}
       \begin{split}
           u \begin{bsmallmatrix}
       r & s
   \end{bsmallmatrix} A_0~\equiv~ 0~(mod~m)\\ \Rightarrow -2u \begin{bsmallmatrix}
       r & s
   \end{bsmallmatrix}A_0~=~2m \begin{bsmallmatrix}
       p & q
   \end{bsmallmatrix} I_2\\ \Rightarrow -u R A_0~=~m P I_2\\ \Rightarrow -u R A_0~=~ P(A_2-eI_2) A_0\\ \Rightarrow P (A_2-eI_2)~=~ -uR\\ \Rightarrow uR~+~P A_2~=~ eP .
       \end{split}
   \end{equation*} 
   As Equation (\ref{pf}) is satisfied, so $\widehat{R}$ is the required matrix such that $\widehat{R}\widehat{A}~=~\widehat{B}\widehat{R}$ and 
   Consequently, $\widehat{R}$ corresponds to a palindromic automorphism of $F_3$.

\end{proof}
\begin{corollary}
    If the matrices $\begin{bsmallmatrix}
        e & 2r & 2s \\ 0 & a & b\\ 0 & c & d
    \end{bsmallmatrix}$ and $\begin{bsmallmatrix}
        e & 0 & 0 \\ 0 & a & b\\ 0 & c & d
    \end{bsmallmatrix}$ satisfy the condition of Theorem \ref{sim1}, then their centralizers are conjugate.  
\end{corollary}
\begin{theorem}\label{sim2}
    Let $\widehat{A},~\widehat{B}~\in ~ GL_3(\mathbb{Z})$ corresponding to two reducible palindromic automorphisms $A~,B$ respectively and of the form $\begin{bsmallmatrix}
        a & b & 2r \\ c & d & 2s \\ 0 & 0 & e
    \end{bsmallmatrix}$ and $\begin{bsmallmatrix}
        a & b & 0 \\ c & d & 0 \\ 0 & 0 & e
    \end{bsmallmatrix}$ respectively and $e~=~\pm 1$. Let $g(t)~=~t^2-\tau  t+ \delta$ be the characteristic polynomial of $A_2~=~\begin{bsmallmatrix}
        a & b \\ c & d
    \end{bsmallmatrix}$ and $\delta~=~\pm 1$.  Let $m~=~ e\tau-1-\delta$ and $A_0=A_2-(\tau - e)I_2.$ Now $\widehat{A},~\widehat{B}$ are conjugate if and only if $\begin{bsmallmatrix}
        r & s
    \end{bsmallmatrix} A_0 ~\equiv~ 0~(mod~ m)$ and $|\tau|~\neq 2$.
\end{theorem}
\begin{proof}
    The proof is similar to the proof of Theorem $\ref{sim1}$.
\end{proof}
\begin{corollary}
    If the matrices $\begin{bsmallmatrix}
        a & b & 2r \\ c & d & 2s \\ 0 & 0 & e
    \end{bsmallmatrix}$ and $\begin{bsmallmatrix}
        a & b & 0 \\ c & d & 0 \\ 0 & 0 & e
    \end{bsmallmatrix}$ satisfy the condition of Theorem \ref{sim2}, then their centralizers are conjugate.
\end{corollary}

We conclude this section by presenting an example of two reducible palindromic automorphisms that are not conjugate in $\widehat{GL_3(\mathbb{Z})}$, since they do not fulfill the necessary condition derived in Theorem \ref{sim1}.
\begin{example}
    Let $\widehat{A}=\begin{bsmallmatrix}
        1 & 2 & 2 \\ 0 & 3 & 4 \\ 0 & 2 & 3
    \end{bsmallmatrix}$ and $\widehat{B}=\begin{bsmallmatrix}
        1 & 0 & 0 \\ 0 & 3 & 4 \\ 0 & 2 & 3
    \end{bsmallmatrix}$ be two reducible palindromic automorphism matrices. If possible suppose there exists a matrix $\widehat{R} \in \widehat{GL_3(\mathbb{Z})}$ such $\widehat{R}\widehat{A}\widehat{R}^{-1}=\widehat{B}$ (see \cite{Onishi}). Assume
    \begin{center}
      $\widehat{R}=\begin{bmatrix}
        a_{11} & a_{12} & a_{13} \\ 0 & a_{22} & a_{23} \\ 0 & a_{32} & a_{33}
    \end{bmatrix}.$  
    \end{center}
    Computing both sides of the equation $\widehat{R}\widehat{A}=\widehat{B}\widehat{R}$, and comparing the corresponding entries, we obtain the following equations:

    \begin{equation}
        \label{eq1}
        2a_{11}+2a_{12}+2a_{13}=0
    \end{equation}
    \begin{equation}
       \label{eq2}2a_{11}+4a_{12}+2a_{13}=0 
    \end{equation}
    
  By solving the system \eqref{eq1}–\eqref{eq2}, we obtain $a_{12}=0$ and $a_{11}=-a_{13}$. Consequently, the first row of $\widehat{R}$ consists entirely of even entries when $a_{11}$ is even, and contains one even and two odd entries when $a_{11}$ is odd. In either case, $\widehat{R}$ does not correspond to a palindromic automorphism, as follows from Theorem~\ref{ncase}. Moreover, the matrices $\widehat{A}$ and $\widehat{B}$ do not satisfy the condition stated in Theorem~\ref{sim1}.\qed
\end{example}

%
%


\begin{thebibliography}{19}
\bibitem{Bh} Sushil Bhunia, Anupam
Singh, {\it Conjugacy classes of centralizers in unitary
groups}, J. Group Theory, 22 (2019), 231-151.

\bibitem{Bh1} Sushil Bhunia, Anupam
Singh, {\it  $z$-classes in groups: a survey}. Indian Journal of Pure and Applied Mathematics. 52, (2021). PP. 713-720.

\bibitem{njf}Neil ~J. Fullarton, {\it A generating set for the palindromic Torelli group}, Algebr. Geom. Topol. {\bf 15} (2015), no.~6, 3535--3567.
       

\bibitem{Bo} A. Borel and J. de Siebenthal, {\it Les sous-groupes de rang maximum des groupes de Lie clos} (French), Comment. Math. Helv. 23 (1949), 200–221.

\bibitem{bv} Martin Bridson,  Karen Vogtmann, {\it 
Automorphism groups of free groups, surface groups and free abelian groups}. Problems on mapping class groups and related topics, 301–316. Proc. Sympos. Pure Math., 74
American Mathematical Society, Providence, RI, 2006
\bibitem{Collins} D. Collins, {\it Palindromic automorphism of free groups}, in: Combinatorial and Geometric group theory, in: London Mat. Soc. Lecture Note Ser., Vol 204. Cambridge Univ. Press, Cambridge, 1995, PP. 63-72.

\bibitem{KG} Valeriy G. Bardakov, Krishnendu Gongopadhyay, Mahender Singh, {\it  Palindromic automorphisms of free groups}. Journal of Algebra. 438, (2015). PP. 260-282.

\bibitem{Bogopolski} Oleg Bogopolski,  Introduction to group theory. Series: EMS textbooks in mathematics. European Mathematical Society; 2008. ISBN: 9783037190418.

\bibitem{GS} Shripad Garge, Anupam Singh, {\it Finiteness of $z$-classes in reductive groups}.
J. Algebra 554 (2020), 41–53.

\bibitem {Glover}
Henry H. Glover, Craig A. Jensen, {\it Geometry for palindromic automorphism groups of free groups}. Comment. Math. Helv. 75 (2000), 
644-667.

\bibitem{Go} Krishnendu Gongopadhyay,  {\it The $z$-classes of quaternionic hyperbolic isometries}, J. Group Theory 16 (2013), 941-964.

\bibitem{Kulkarni} Ravi S Kulkarni, {\it Dynamics of linear and affine maps}. Asian J. Math. 12 (2008). no. 3. PP. 321-344.

\bibitem{Kulkarni1} Ravi S Kulkarni, {\it Dynamical types and conjugacy classes of centralizers in
groups}.  J. Ramanujan Math. Soc. 22 (2007), no. 1, PP 35-56.

\bibitem{Ku3} Ravindra Kulkarni,  Rahul Dattatraya Kitture, Vikas S. Jadhav, {\it $z$-classes in
groups}, J. Algebra Appl. 15 (2016), no. 7, 1650131.

\bibitem{center}  A.I. Nekritsuhin, {\it On some properties palindromes of automorphisms of a free group}, Chebyshevskii
Sb. 15 (1) (2014) 141–145 (in Russian).

\bibitem{parsad}Shiv, Parsad, 
{\it z-classes in the mapping class group}.
Topology Appl.318 (2022), Paper No. 108196, 5 pp.

\bibitem{PR} Adam Piggott and Kim  Ruane, {\it Normal forms for automorphisms of universal Coxeter groups and palindromic automorphisms of free groups}. 
Internat. J. Algebra Comput.20(2010), no.8, 1063–1086.
\bibitem{Onishi} Harry Appelgate and Hironori Onishi, {\it The similarity problem for $3\times 3$ integer matrices}. 
Linear Algebra and Applications. 42: 159-174 (1982).

\bibitem{St} Robert Steinberg, {\it Conjugacy Classes in Algebraic Groups}, notes by V. Deodhar,
Lecture Notes in Mathematics 366, Springer-Verlag (1974).
\end{thebibliography}
\end{document}